\documentclass{amsart}

\usepackage{graphicx}%
\usepackage{multirow}%
\usepackage{amsmath,amssymb,amsfonts}%
\usepackage{amsthm}%
\usepackage{mathrsfs}%
\usepackage[title]{appendix}%
\usepackage{xcolor}%
\usepackage{textcomp}%
\usepackage{manyfoot}%
\usepackage{booktabs}%
\usepackage{algorithm}%
\usepackage{algorithmicx}%
\usepackage{algpseudocode}%
\usepackage{listings}%
\usepackage{longtable}
\usepackage{nicematrix,tikz-cd}

\usepackage{natbib}
\setcitestyle{authoryear,open={(},close={)}} 

\newtheorem{theorem}{Theorem}[section]
\newtheorem{proposition}[theorem]{Proposition}

\newtheorem{corollary}[theorem]{Corollary}

\newtheorem{question}[theorem]{Question}

\theoremstyle{definition}
\newtheorem{definition}{Definition}
\newtheorem{example}[theorem]{Example}

\newtheorem{remark}[theorem]{Remark}

\numberwithin{equation}{section}

\begin{document}
\title[Real gamma distribution on analytic bundles]{Real gamma distribution on analytic bundles of flag varieties}
\author[Haoming Wang]{Haoming Wang}
\address{Center for Combinatorics, Nankai University, Tientsin {\rm300071}, China}
\email{wanghm37@nankai.edu.cn}
\thanks{}
\date{}
\dedicatory{}

\begin{abstract}
This paper introduces four matrix normal distributions on analytic bundles of flag varieties, extending the separable covariance $\varPhi \otimes \varPsi$ with potentially variable-level ($\varPsi$) and/or sample-level ($\varPhi$) correlations. The joint distribution of sample variances and covariances, leading to the product moment distribution, is considered when precision matrices admit a specific form. Several well-known consequences, including the non-central Wishart distribution and normal quadratic forms, now appear as corollaries.
\end{abstract}

\maketitle

\section{Introduction}

In standard textbooks on probability and statistics, such as \cite{gupta2018matrix} and \cite{mathai2022mul}, the 1-dimensional real gamma distribution is defined as a probability distribution whose density is given by
\[f_{\alpha,\beta}(x) = \frac{\beta^{\alpha}}{\varGamma(\alpha)}x^{\alpha -1} e^{- \beta x}, \quad \alpha>0, \, \beta>0,\]
where $\varGamma(\alpha)$ is known as the gamma function, or the probability mass to the above integral. Equivalently, one can also define $\varGamma(n)$ as Stirling's notion $(n-1)!$ and use the method of analytic continuation to derive the same distribution.

An analogue of the multivariate real gamma distribution involves two parameters $A$ and $B$, where $A$ is a positive real number and $B$ is a $k\times k$ real symmetric positive definite matrix. However, the classic analytic continuation method becomes problematic when the distribution has a non-zero expectation. See, Problem 6.10 in \cite{eaton2007multivariate} and Chapter VII in \cite{faraut1994analysis}. \cite{eaton2007multivariate} conjectured that all possible values of $A$ should lie in the following set on the real half-line
\[\left\{0,\frac{1}{2},1,\frac{3}{2},\dots,\frac{k-1}{2}\right\}\bigcup \left[\frac{k-1}{2},\infty\right).\]
This problem has been resolved by \cite{peddada1991}, \cite{letac2018laplace}, \cite{mayerhofer2019wishart} using different methods, and this set is named after \cite{gindikin1964analysis}. However, these results only concerned a special case of the real gamma distribution on a relatively narrow Grassmannian, studied by \cite{james1974generalized}, \cite{grossandrichards1987trans}, \cite{Shimura1990}, \cite{beerends1993certain}, \cite{Kazhdan1993I-IV}, \cite{johansson1997random,johansson2001discrete}, \cite{grinberg2004radon}, \cite{Postnikov2005affine}, \cite{Bernig2011aom}, \cite{Doubrov2018}, \cite{hedenmalm2021planar}, \cite{Galashin2022flag}, \cite{Kimura2024hyper}, and \cite{kubo2025intertwining}, namely the Wishart distribution.

The real gamma distribution on the affine Grassmannian ${\rm Gr}_{n,k}$ was first studied by \cite{wishart1928generalised}, \cite{ingham1933integral}, \cite{herz1955}. For a detailed discussion with ${\rm Gr}_{n,k}$, see \cite{goodman2009gtm255}, \cite{lakshmibai2015grassmannian,lakshmibai2001flag}, their counterparts in integral geometry, \cite{faraut1994analysis}, \cite{helgason2001differential}, \cite{gel2014integral}, and in multivariate statistics, \cite{muirhead1982aspects}, \cite{mathai2022mul}. In fact, ${\rm Gr}_{n,k}$ is a projective variety under the Zariski topology on $\mathbb{P}^{\binom{n}{k}-1}$ with defining ideal given by Pl\"ucker polynomials
\[\sum_{h=1}^{k+1} (-1)^{h} p_{i_1,\dots,i_{k-1},j_h} p_{j_1,\dots,j_{h-1},\widehat{j_h},j_{h+1},\dots,j_{k+1}},\]
where $p_{\underline{i}}:{\rm Gr}_{n,k} \to \mathbb{P}^{{n \choose k}-1}$ and $\underline{i} = (i_1,i_2,\dots,i_k), 1\leq i_1 < i_2 < \dots < i_k \leq n$ is the determinant representation of ${\rm Gr}_{n,k}$. This variety has dimension $k(n-k)$ and one can show further that it is compact, Tychonoff ($T_0$), irreducible, and generated by a minimal prime ideal. The real gamma distribution arises naturally as the distribution of a sample covariance matrix from independent identically distributed normal vectors. For example, if $x_1,x_2,\dots,x_n \sim_{i.i.d.} N(\mu,\Sigma)$ where $\mu$ is $k\times 1$ and $\Sigma$ is $k\times k$ positive definite, \cite{student1908}, \cite{cochran1946relative}, and \cite{dykstra1970establishing} proved that the sample covariance matrix $S^2$ given by
\[\begin{aligned}
    \frac{n-1}{n} S^2 = \sum_{i=1}^n (x_i - \bar{x})(x_i - \bar{x})',
\quad    n\bar{x} = \sum_{i=1}^{n} x_i,
\end{aligned}\]
follows the real gamma distribution with parameters $\frac{1}{2}(k+1)$ and $\frac{1}{n}\Sigma^{-1}$ on ${\rm Gr}_{n,k}$, the former a positive integer or half-integer which belongs to the Gindikin set. However, its analytic continuation still remains unknown until \cite{letac2018laplace} and \cite{mayerhofer2019wishart} give a complete answer independently.

Recent progress on Grassmannians was mostly motivated by the flags and vector bundles, which provided natural isomorphisms between Stiefel manifolds and Grassmannians. For instance, equip each $k$-dimensional subspace in ${\rm Gr}_{n,k}$ with an orthogonal frame $V \cong \mathbb{R}^k, \mathbb{C}^k,$ or $\mathbb{H}^k$, and we will obtain a local trivialisation on each chart $(U_i,\phi_i)$ via a homeomorphism $\pi$. 
\begin{figure}[htbp]
    \centering
    \begin{tikzcd}[row sep=large, column sep=large] \pi^{-1}(U_i) \arrow[r, "\phi_i"] \arrow[d, "\pi"'] & U_i \times V \arrow[d, "\rm{proj}^1"] \\ U_i \arrow[r, ""] & U_i  \end{tikzcd} \quad \begin{tikzcd}[row sep=large, column sep=large] {\rm St}_{n,k} \arrow[r, "f"] \arrow[d, "\pi"] & {\rm St}_{n,k} \arrow[d, "\sigma"] \\ {\rm Gr}_{n,k} \arrow[r, "g"] & {\rm Flag}{(k;\mathbb R^k)}  \end{tikzcd}
    \label{fig:tikzcd1}
\end{figure}
A chain of increasing subspaces in $V$ is a flag
\[\{0\}=V_{0}\subset V_{1}\subset V_{2}\subset \cdots \subset V_{l}=V,\]
which is uniquely labelled by an integer partition $\lambda = (\lambda_1,\lambda_2,\dots,\lambda_l)$ with $ n\geq \lambda_1>\lambda_2>\dots>\lambda_{l}\geq0$ in this sense, $\dim (V_{i}) = \lambda_{l+1 - i}$. Since the classification of flag varieties is non-trivial, it poses challenges for the generalisation of Student's theorem on them.

Despite this progress, the generalization of such distributions to more general flag varieties and to matrix normal distributions with heterogeneous covariance structures remains challenging. In this paper, we explore statistical distributions on Grassmannians and their flag bundles, aiming to classify double flags via the ampleness of such a matrix normal family and its generalisations,
\begin{equation}
    \begin{aligned}
        g_{n,k}({X})  = \frac{1}{(2\pi)^{\frac{nk}{2}}| \varPhi|^{\frac{k}{2}}| \varPsi|^{\frac{n}{2}}} {\rm etr}\left[ - \frac{1}{2}\varPhi^{-1}({X} -  M) \varPsi^{-1}({X} -  M)^{\prime}\right],
    \end{aligned}
    \label{eq: dawid matrix normal distribution}
\end{equation}
where ${\rm etr}(\cdot) = \exp {\rm tr}(\cdot)$, $ \varPhi$ and $ \varPsi$ are $n\times n$ and $k \times k$ real symmetric positive definite matrices, respectively, and $ M$ is an arbitrary $n \times k$ real matrix. Before that, we study the density of its sample covariance and calculate the characteristic function and distribution of characteristic roots. However, after the matrix transformation $Y= X\varPsi^{-\frac{1}{2}}= (y_1,y_2,\dots,y_k)$, it reduces to the independent identically distributed samples. Thus, instead of considering a common tensor matrix $A = \varPhi^{-1}$ in quadratic forms $y_iAy_j$ ($i\leq j$), we introduce $\frac{1}{2}k(k+1)$ heterogeneous quadratic forms
\begin{equation}
    \begin{aligned}
        {y}_1'{A}_{11}{y}_1, {y}_1'{A}_{12}{y}_2, \dots, {y}_1'{A}_{1k}{y}_k,\\
        {y}_2'{A}_{22}{y}_2, \dots, {y}_2'{A}_{2k}{y}_k, \\
        \vdots\qquad \\
        {y}_k'{A}_{kk}{y}_k,
    \end{aligned}
    \label{eq: illustration of triangular ararry heter}
\end{equation}
where $A_{ij}^{\prime} =  A_{ji}$ represents the inverted covariance matrix between the $i$th and the $j$th variables. This corresponds to four tensor forms $T_{1}$, $T_{1\frac{1}{2}}$, $T_{2}$, and $T_{3}$ of the population precision matrix that will be seen in align with the elliptical contoured notations introduced by \cite{1990Generalized}. As pointed out by them, the problem for the determination of spectral decomposition remains an open question.

\section{Analytic bundles of flag varieties}

Assume the readers are familiar with basic combinatorial notions, like integer partitions, Young tableaux, skew tableaux, and Schur-Weyl duality, etc. For an overview of these topics, see \cite{fulton1997young}, \cite{macdonald1998symmetric}, \cite{goodman2009gtm255}. We often identify the irreducible representations of the general linear group $GL_k$ and the symmetric group $S_k$ via the Schur-Weyl duality, whose proof can be found above. Let $V = \mathbb K^n,$ where $ \mathbb K = \mathbb R, \mathbb C, \mathbb H.$

\begin{proposition} ${\rm End}_{S_k}(V^{\otimes k}) \cong \mathbb{C}[GL(V)]$ and ${\rm End}_{GL(V)}(V^{\otimes k}) \cong \mathbb{C}[S_k]$. 
\end{proposition}

An analytic bundle is a vector bundle such that the projection mapping from the bundle to the base is analytic. The following three examples deal with the real case, $\mathbb K = \mathbb R$.

\begin{example}
Consider the Stiefel manifold $\mathrm{St}_{n,k}$ of orthonormal $k$-frames in $\mathbb{R}^n$ and the Grassmann manifold $\mathrm{Gr}_{n,k}$ of $k$-dimensional subspaces. There is a natural projection 
\[
\mathbb{R}^k \to \mathrm{St}_{n,k} \overset{\pi}{\to} \mathrm{Gr}_{n,k}
\]
sending a $k$-frame to the subspace it spans. This is an analytic fiber bundle whose fiber over each point is isomorphic to the orthogonal group $O_{k}$. Specifically, for a fixed $k$-dimensional subspace $V \in \mathrm{Gr}_{n,k}$, any orthonormal basis of $V$ is related to any other by an orthogonal transformation, so the fiber $\pi^{-1}(V)$ is exactly the set of orthonormal bases of $V$, which is a tensor for $O_{k}$. Therefore, $\mathrm{St}_{n,k}$ is a principal $O_{k}$-bundle over $\mathrm{Gr}_{n,k}$, and the projection $\pi$ is analytic.
\end{example}

\begin{example}
Let $\lambda = (\lambda_1, \lambda_2, \dots, \lambda_l)$ be an integer partition with $n \geq \lambda_1 > \lambda_2 > \cdots > \lambda_l \geq 0$. The flag variety $\mathrm{Flag}(\lambda; \mathbb{R}^n)$ consists of all flags 
$0 \subset V_1 \subset V_2 \subset \cdots \subset V_l \subset \mathbb{R}^n$ such that $\dim V_i = \lambda_{l+1-i}$. The orthogonal group $O_n$ acts transitively on $\mathrm{Flag}(\lambda; \mathbb{R}^n)$, and the stabilizer of a fixed flag 
is the block-diagonal subgroup $O_{\lambda}$,
\[ \begin{bNiceMatrix}
		  O_{\lambda_1 - \lambda_2} & & & \dots & O\\
		& O_{\lambda_2 - \lambda_3} & & \\
        & & \ddots &  & \vdots\\
		\vdots & & & O_{\lambda_{l-1} - \lambda_l} & \\
		O & \dots & & &  O_{\lambda_l}
		\CodeAfter
		\tikz \draw[dashed] (1-|2) |- (2-|3) |- (3-|4) |- (4-|5) |- (5-|6);
		\tikz \draw[dashed] (2-|1) -| (3-|2) -| (4-|3) -| (5-|4) -| (6-|5);
	\end{bNiceMatrix}.\]
This gives rise to an analytic fiber bundle, expressed as
\[
O_\lambda \longrightarrow O_n \overset{\pi_\lambda}{\longrightarrow} \mathrm{Flag}(\lambda; \mathbb{R}^n),
\]
where $\pi_\lambda$ is the natural analytic projection. In particular, $O_n$ is an analytic principal $O_\lambda$-bundle over 
$\mathrm{Flag}(\lambda; \mathbb{R}^n)$. When $\lambda = (n-k, k)$, the flag variety reduces to the Grassmannian $\mathrm{Gr}_{n,k}$, and we recover the familiar bundle 
$O_k \times O_{n-k} \to O_n \to \mathrm{Gr}_{n,k}$. The Stiefel manifold $\mathrm{St}_{n,k}$ then arises as 
the associated bundle $O_n \times_{G} \mathbb{R}^{n-k}$, where $G=O_{n-k}$ acts trivially on $\mathbb{R}^{n-k}$.
\end{example}

The following example is important in our main Theorem \ref{thm: flag}.
\begin{example}
Let $A \in \mathbb{R}^{n \times k}$ be a matrix of full column rank ${\rm rank}(A) = k \le n$. Consider the graph of $A$
\[
\Gamma_A = \{ (x, y) \in \mathbb{R}^k \oplus \mathbb{R}^n \mid y = Ax \},
\]
which is an $k$-dimensional subspace of $\mathbb{R}^{n+k}$ and therefore corresponds to a point in the Grassmannian ${\rm Gr}_{n+k,k}$. Let $\{v_1, \dots, v_k\}$ be an orthonormal basis of $\mathbb{R}^k$ satisfying $A'Av_j = \sigma^2_j v_j$. This basis defines ${\rm Flag}(\mathbb R^k)$, a complete flag $0 \subset \langle v_1 \rangle \subset \langle v_1, v_2 \rangle \subset \cdots \subset \mathbb{R}^k$. Since the projection $\pi_1: \Gamma_A \to \mathbb{R}^k$, $(x, Ax) \mapsto x$ is an isomorphism, we can lift ${\rm Flag}(\mathbb R^k)$ to $\Gamma_A$ by defining
\[
\Gamma_j = \{ (x, Ax) \mid x \in \langle v_1, \dots, v_j \rangle \} \subset \Gamma_A, \qquad j = 1, \dots, k.
\]
Then $\dim \Gamma_j = j$ and $\Gamma_j \subset \Gamma_{j+1}$, yielding a complete flag on $\Gamma_A$,
\[
0 \subset \Gamma_1 \subset \Gamma_2 \subset \cdots \subset \Gamma_k = \Gamma_A.
\]
This yields an embedding
\[
\iota: \mathrm{Flag}(\mathbb{R}^k) \rightarrow \mathrm{Flag}(1,2,\dots,k; \mathbb{R}^{n+k}),
\]
which sends the right flag of $A$ to a flag on its graph. The construction is analytic and depends only on the right flag since the left flag is not encoded in the image.
\end{example}

\begin{proposition}
    Let $V$ a real analytic compact connected projective variety irreducible in the affine sense. Let $\lambda \vdash k$ be an integer partition and $n \ge k$. Then the following four objects are equivalent descriptions for $V$,
    $$\mathrm{Flag}(\lambda; \mathbb{R}^n), \quad O_n/O_\lambda, \quad {\rm St}_{n,k}/O_\lambda, \quad  SO_n/O_\lambda.$$
\end{proposition}
Its proof can be found in standard textbooks, for instance, \cite{fulton1997young,lakshmibai2001flag,goodman2009gtm255}.

\section{Hypergeometric function}

The hypergeometric function in one variable is defined as the power series 
\[{}_{p}F_{q}\left(\begin{matrix}
    	    \underline{a}; \underline{b}\end{matrix}; z\right) = \sum _{n=0}^{\infty }{\frac {(a_1)^{n}(a_2)^{n}\cdots(a_p)^n}{(b_1)^{n}(b_2)^{n}\cdots(b_q)^n}}{\frac {z^{n}}{n!}},\]
where $\underline{a} = (a_1,a_2,\dots,a_p)$, $\underline{b} = (b_1,b_2,\dots,b_q)$, and $(q)^n =q(q+1)\cdots (q+n-1)$ denotes the rising Pochhammer symbol. For example, the Gauss hypergeometric function ${}_{2}F_{1}(a,b;c;z)$ terminates if either $a$ or $b$ is a non-positive integer, in which case the function reduces to a polynomial
\[{}_{2}F_{1}(-m,b;c;z)=\sum _{n=0}^{m}(-1)^{n}{\binom {m}{n}}{\frac {(b)^{n}}{(c)^{n}}}z^{n}.\]
The hypergeometric function has many nice properties, which tie closely to other mathematical branches, like complex analysis, partial differential equations, group theory, and combinatorics etc. First, it can also be defined as a solution to the Laplace-Beltrami equation, which subjects to a second-order elliptic differential operator. This is a famous recursive relation
    \begin{eqnarray*}
            \lim\limits_{c\to \infty}{}_{p+1}F_{q} (\underline{a},c;\underline{b};c^{-1}z) = {}_{p}F_{q} (\underline{a};\underline{b};z), \\
        \lim\limits_{c\to \infty}{}_{p}F_{q+1}(\underline{a};\underline{b},c;cz)  = {}_{p}F_{q} (\underline{a};\underline{b};z).
    \end{eqnarray*}

In the multivariate case, these hypergeometric functions are generalised as the series of Gegenbauer, or ultraspherical polynomials. The related representation had been settled in the last century, for example, \cite{szego1954singularities}, \cite{Nehari1956OnTS}, \cite{stein1971}, \cite{Zayed1980OnTS}, \cite{ebenfelt1998}, while for the hypergeometric function with matrix arguments, it is defined in terms of the eigenvalues of a Hermite matrix. \cite{herz1955} might been seen as the first to use this definition by calculating a Laplace transformation over orthogonal groups and its inverse. \cite{constantine1963some,constantine1966hotelling} gave the zonal polynomials as a basis for these hypergeometric functions. \cite{Vilenkin2010RepresentationOL} generalised these results to a family of Jack polynomials, but for general Macdonald polynomials or Koornwinder polynomials, this field still lacks exploration.

In this paper, we do not go through the representation part of hypergeometric functions. Instead, we will focus on the related formulas of gamma and beta integrals that will be given as follows. The multivariate Gamma function, denoted by $\varGamma_{p}(a)$, is defined to be
\begin{equation}
    \varGamma_{n}(a) = \int_{A>0} \operatorname{etr} (-A) |A|^{a-\frac{n+1}{2}} (dA),
\end{equation}
where $\Re (a) > \frac{1}{2}(n-1)$ and $A>0$ means the integral is taken over the space of real symmetric positive definite $n \times n$ matrices. The multivariate Beta function, denoted by $B_{n}(a,b)$, is defined to be
\begin{equation}
    B_{n}(a,b) =\int_{0< X < I} |X|^{a - \frac{n+1}{2}} |I-X|^{b - \frac{n+1}{2}} (dX),
\end{equation}
where $\Re (a), \Re (b) > \frac{1}{2}(n-1)$, and the integral is taken over all $n\times n$ real symmetric matrices $X$ such that both $X$ and $I - X$ are positive definite. The multivariate Beta function is related to the multivariate Gamma function by the formula 
\begin{equation}
    B_{n}(a,b) = \frac{\varGamma_{n}(a)\varGamma_{n}(b)}{\varGamma_{n}(a + b)}.
\end{equation}

It's notable to mention that the hypergeometric function of a matrix argument has a similar recursive relation due to \cite{herz1955}. We are only interested in the gamma and beta integrals, whose proof are given by \cite{constantine1963some}.
	\begin{eqnarray}
	    \begin{aligned}\int_{X>0} \operatorname{etr}(-XT) |X|^{c - \frac{n+1}{2}} &{}_{p}F_{q} (\underline{a};\underline{b}; X) (dX)\\		= \varGamma_{n}(c) |T|^{-a} &{}_{p+1}F_{q} (\underline{a},c;\underline{b}; T^{-1}) 	\end{aligned}\\
	    \begin{aligned}
    	\int_{0< X < I} |X|^{c - \frac{n+1}{2}} |I-X|^{d - \frac{n+1}{2}} {}_{p}F_{q} \left(\underline{a};\underline{b}; TX\right) (dX)\\
		= {B_{n} (d-c,c)}^{-1}{}_{p+1}F_{q+1} \left(c,\underline{a};d,\underline{b}; T\right),\\
	\end{aligned}
        \label{eq: hypergeometric function1}
	\end{eqnarray}
where $\Re (d-c), \Re (c) > \frac{1}{2}(n-1)$ and $T$ is an $n \times n$ real symmetric positive definite matrix. Besides, the hypergeometric function of two matrix arguments is defined by the zonal polynomial, a topic that can be found in \cite{hua1963harmonic}, \cite{muirhead1982aspects}, \cite{Vilenkin2010RepresentationOL}, \cite{macdonald1998symmetric}, \cite{shimizu2022}, \cite{wang2025matrix}. We are not going to give its definition here. However, when two matrix arguments are of the same size, we have 
\begin{equation}
    \int_{O_n} {}_{p}F_{q}\left(\underline{a};\underline{b}; X H Y H^{\prime}\right) [dH] = {}_{p}F_{q}\left(\underline{a};\underline{b}; X, Y\right),
\end{equation}
where $ X$ and $ Y$ are both $n \times n$ symmetric matrices, the integral is taken with respect to the normalised Haar measure $[dH]$ over the space of $n\times n$ orthogonal matrices $O_n$. 

If two matrices are of unequal size, a similar formula can be found in the following lemma, whose proof is standard and based on the property of zonal polynomials for block matrices, as in \cite{SHIMIZU2021104714}. Let $n,k$ be positive integers and $n\geq k$. 

\begin{proposition}\label{lem: hypergeometric} Suppose $ A$ is an $n\times n$ and $ B$ is a $k\times k$ symmetric matrix. 
    \[\begin{aligned}
        \int_{O_n}{}_{0}F_{0}( A H_1 B H_1^{\prime}) [dH] = {}_{0}F_{0}( A, B),\\ 
         H = (H_{1},  H_{2}), \text{ where }  H_1 \text{ is } n \times k,
    \end{aligned}\]
    that is, the integral runs over the orthogonal group $O_n$.
\end{proposition}

The following lemma by \cite{james1961zonal} will be used to derive the non-central distribution.
\begin{proposition}\label{lem: james55} Suppose $X$ is an $n\times k$ matrix.
    \[\begin{aligned}
        \int_{O_n} \operatorname{etr}( X H_1^{\prime}) (d H) = {}_{0}F_{1}\left(\frac{1}{2}n;\frac{1}{4} X' X\right), \\ 
         H = (H_{1},  H_{2}), \text{ where }  H_1 \text{ is } n \times k,
        \end{aligned}\]
    and the integral runs over the orthogonal group $O_n$ similar to Lemma \ref{lem: hypergeometric}.
\end{proposition}

The following lemma due to \cite{khatri1966} will be used to prove Theorem \ref{thm: main theorem}.
	\begin{proposition} Let $ A$ be an $n \times n$ symmetric matrix and $ B$ an $k \times k$ symmetric positive definite matrix with $n \geq k$. Let $ X$ be an $n \times k$ matrix. Then
		\begin{equation*}
			\int_{ X^{\prime}  X= S} \operatorname{etr} ( A  X  B  X^{\prime}) (d  X)= \frac{\pi^{\frac{nk}{2}}}{\varGamma_{p}(\frac{n}{2} )} | S|^{\frac{n-k-1}{2}} {_0}F_0( A,  B S).\label{eq: Khatri1966}
		\end{equation*}\label{lem: Khatri66}
	\end{proposition}

The incomplete Gamma and Beta integrals involving hypergeometric functions are given by \cite{constantine1963some} and \cite{Davis1979InvariantPW,Davis1981OnTC}. We will use them to derive the extreme eigenvalue distribution.
\begin{proposition}
\label{lem: davis} Let $R$ be an $p \times p$ real symmetric positive definite matrix. 
\begin{equation}
\begin{aligned}
\int_{0}^{ R} \operatorname{etr}(- A S)| S|^{c-\frac{p+1}{2}} & {}_{p}F_{q}\left(\underline{a};\underline{b}; B  S\right) d  S \\
                = \frac{| R|^a }{{B}_{p}(c,\frac{p+1}{2})}
                & {}_{p+1}F_{q+1} \left(
            c,\underline{a};c + \tfrac{1}{2}(p+1),\underline{b}; -  A  R, B  R\right).
\end{aligned}
\end{equation}
\begin{equation}
        \begin{aligned}
        \int_{0}^{ R} | S|^{c-\frac{p+1}{2}} | I -  S|^{d - \frac{p+1}{2}} & {}_{p}F_{q}\left(\underline{a};\underline{b}; A  S\right) d  S \\ = \frac{| R|^a }{{ B}_{p}\left(c,\frac{p+1}{2}\right)} {}_{p+2}F_{q+1} & \left(
            c, - d + \tfrac{1}{2}(p+1), \underline{a}; c + \tfrac{1}{2}(p+1), \underline{b};  A,  A  R\right).
    \end{aligned}
\end{equation}
\end{proposition}

Last but not least, we will introduce an important matrix decomposition. Assume that $Z$ is a $n \times k$ real matrix with full column rank. Thus, $Z$ can be uniquely decomposed into $Z = HR^{\frac{1}{2}},$ where $H$ is an $n\times k$ matrix satisfying $H'H = I_k$ and $R$ is a $k\times k$ symmetric positive definite matrix. In particular, we can choose $H = Z(Z'Z)^{-\frac{1}{2}}$ and $R = Z'Z$.

The set of $n\times k$ matrices $H$ satisfying $H'H = I_k$ is said to be a Stiefel manifold, denoted by ${\rm St}_{n,k}$. For any such $H\in {\rm St}_{n,k}$, we can always expand it into an orthogonal matrix $K = (H,H_\perp) \in O_n$, where the column vectors of $H_\perp \in {\rm St}_{n,n-k}$ are orthogonal to the column vectors of $H\in V_{n,k}$. \cite[Equation 8.19]{james1954} was the first to prove that
    \begin{equation*} \bigwedge_{i,j=1}^n dz_{ij} = 2^{-k} |Z'Z|^{\frac{1}{2}(n-k-1)} \bigwedge_{i\le j}^n d(z_i'z_j) \bigwedge_{i< j}^n h_i'dh_j \bigwedge_{i=k+1}^n \bigwedge_{j=1}^k h_i'dh_j, \label{eq: james1954}\end{equation*}
    where $H = (h_1,\dots,h_k), H_\perp = (h_{k+1},\dots,h_n)$. In matrix form,
\begin{proposition}\label{prop: herz55}
$(dZ) = 2^{-k} |R|^{\frac{1}{2}(n-k-1)} \cdot (dR) \cdot (dK)$, where $(dK) = {(H'dH)} \cdot {(H_\perp'dH)}$.
\end{proposition}
The geometric meaning of $(dK)$ is, \({(H'dH)} \) describes the rotation of the column vectors of the object \( H \) within the \(k \)-dimensional subspace it spans, thus corresponding to the volume element of $O_k$, and \( {(H_\perp'dH)} \) describes the change of direction of the \(k\)-dimensional subspace in the \( n \)-dimensional space, that is, the movement of the subspace in \({\rm Gr}_{n,k} \). The two together describe the Stiefel manifold \( {\rm St}_{n,p} \), which is a principal bundle on the Grassmannian ${\rm Gr}_{n,k}$ with fibre \(O_k\).

\section{Matrix normal distribution}



Let $M$ be a complete Riemannian manifold that can be embedded into the space of rectangular matrices $M_{n,k}\cong \mathbb{R}^{n\times k}$ for appropriate $n,k$. The set of all probability measures on $M$, denoted as $\mathcal{P}(M)$, is a 
topological space under weak convergence, on which a topological basis is given by 
\[U_{\underline{f};\varepsilon}(\mu) = \bigcap_{i=1}^m\left\{\nu \in \mathcal{P}(M):|\mu-\nu|(f_i)< \varepsilon\right\},\\ f_{1},\dots,f_{m} \in C_b(M), \varepsilon\geq 0.\]
In fact, we have the proposition whose proof can be found in \cite[Section 5]{billingsley2013convergence} and \cite[Theorem 17.23]{kechris2012classical}, 
\begin{proposition}
    $\mathcal{P}(M)$ is a complete separable metric space, and among which are absolutely continuous with respect to the Lebesgue measures on $M_{n,k} $ is a closed subspace of $\mathcal{P}(M)$. 
    Let us denote this set as $\mathcal{P}_{0}(M)$. In addition, the convolution of two measures in $\mathcal{P}_0(M)$ turns it into a Polish group.
\end{proposition}

Without loss of generality, we could take $M$ to be
\[\begin{aligned}
    M_{n,k} & = \left\{\text{all } n\times k \text{ matrices}\right\},\\
    {\rm St}_{n,k} & = \left\{\text{all } n\times k \text{ matrices $X$ such that } X'X =I_k\right\},\\
    {\rm Gr}_{n,k} & = \left\{\text{all } k\text{-dimensional subspaces in } \mathbb{K}^n\right\},\\
    {\rm Flag}(\lambda,\mathbb{K}^n) & = \left\{\text{all chains of subspaces encoded by } \lambda \text{ in } \mathbb{K}^n\right\},
\end{aligned}\]
the space of rectangular matrices, the Stiefel manifold of orthogonal-invariant frames, the Grassmannian of $k$-dimensional subspaces, or the flag variety, with $\mathbb{K} = \mathbb R, \mathbb C, \mathbb H.$ But we are only interested in the real gamma distribution in this paper, thereby all entries in matrices should be understood to be real scalars. 

A distribution $p(X) \in \mathcal{P}_0 (M_{n,k})$ is called left-spherical if $p(\gamma X) = p(X)$ for any $\gamma \in O_n$. \cite{1990Generalized} proved that 
\begin{proposition}\label{prop: stiefel}
    There exists only one left-spherical distribution on ${\rm St}_{n,k}$.
\end{proposition}
This distribution coincides with the unique left-Haar measure on the Polish group ${\rm St}_{n,k}$, and we will call it a uniform distribution. Moreover, \cite{1990Generalized} introduced two other classes of spherical distributions, namely the multivariate spherical distribution and vector spherical distribution. To be precise, a distribution $p(X) \in \mathcal{P}_0 (M)$ is called multivariate spherical if $p(\gamma_1 x_1, \gamma_2 x_2, \dots, \gamma_k x_k) = p(x_1,x_2,\dots,x_k)$ for any $\gamma_i \in O_n, i = 1,2,\dots,k$; vector spherical if $p(\varGamma {\rm vec} (X')) = p({\rm vec} (X'))$ for any $\varGamma \in O(nk)$. Note that ${\rm vec} (X')$ is the vector $$x_{11},x_{12},\dots,x_{1k},x_{21},x_{22},\dots,x_{nk}.$$ 
For the elliptical spherical distribution, we shall adopt another definition. Before that, we define a normal matrix to have each of its entries a normal variable and give some special classes. 

\begin{definition} An $n\times k$ matrix $X$ is said to follow a matrix normal distribution with the probability density function $p(X)$
	\begin{itemize}
		\item[$T_{1}$] if there exist orthonormal column vectors $\{ b_{j}\}$ of length $k$ such that
        \begin{equation}
            p(X) = c_{1}\operatorname{etr} \left(-\frac {1}{2} \sum_{j=1}^{k} { A}_{jj}  X{ B}_{jj}{ X}^{\prime} - \sum_{j=1}^{k}\sum_{j^{\prime}=j+1}^{k}{ A}_{jj^{\prime}}  X{ B}_{jj^{\prime}}{ X}^{\prime}\right),
			\label{eq: T1}
		\end{equation}
		where $ B_{jj^{\prime}} =  b_{j} { b_{j^{\prime}}}^{\prime}$, 
        $A_{ij}^{\prime} = A_{ji}$, the diagonal entries in $ A_{jj^{\prime}}$ are positive if $j=j^{\prime}$ and equal to zero if $j\neq j^{\prime}$;
        \item[$T_{1\frac{1}{2}}$] if it is $T_{1}$ and additionally, $ A_{jj^{\prime}} =  A_{jj}\delta_{jj^{\prime}}$ if $j\neq j^{\prime}$,
		\begin{equation}	
        p( X) = c_{1\frac{1}{2}}\operatorname{etr} \left(-\frac {1}{2} \sum_{j=1}^{k} { A}_{jj} X{ B}_{jj}{ X}^{\prime}\right);		\label{eq: T1.5}		\end{equation}

        \item[$T_{2}$] if it is $T_{1}$ and there exist further orthonormal column vectors $\{ a_{i}\}$ of length $n$ such that 
		\begin{equation}
			 p( X) =
             c_2\operatorname{etr} \left(-\frac {1}{2} \sum_{i=1}^{n}\sum_{j=1}^{k} \gamma_{ij} { A}_{i} X{ B}_{j}{ X}^{\prime}\right),
			\label{eq: T2}
		\end{equation}
		where $ A_{i} =  a_{i}{ a_{i}}^{\prime}$ and $ B_{j} =  b_{j}{ b_{j}}^{\prime}$;

        \item[$T_3$] if it is $T_{2}$ and additionally, there exist $\alpha_{i}, \beta_{j}$ such that $\gamma_{ij} = \alpha_{i} \beta_{j}$,
			\begin{equation}
				p( X) = 
                c_3 \operatorname{etr} \left(-\frac {1}{2}  \varPhi^{-1}  X  \varPsi^{-1} {{ X}}^{\prime}\right),
				\label{eq: T3}
			\end{equation}
			where $ \varPhi^{-1} = \sum_{i=1}^{n} \alpha_{i}  a_{i} { a_{i}}^{\prime}$ and $ \varPsi^{-1} = \sum_{j=1}^{p} \beta_{j}  b_{j}{ b_{j}}^{\prime}$.     
	\end{itemize}\label{def: normal matrix}
\end{definition}

\begin{theorem} The constants $c_1$, $c_{1\frac{1}{2}}$, $c_2$, and $c_3$ in Definition 1 and their corresponding precision matrices $\varTheta = \varSigma^{-1}$ where $\varSigma = \operatorname{Cov}(\operatorname{vec}(X'))$, are 
\begin{align}
        & c_1 = \frac{| \varTheta_{1}|^{\frac{1}{2}}}{(2\pi)^{\frac{np}{2}}}, \quad  \varTheta_1 = \sum_{j=1}^{p}\sum_{j^{\prime}=1}^{p}  A_{jj^{\prime}} \otimes  B_{jj^{\prime}}, \\
    & c_{1\frac{1}{2}} = \frac{| \varTheta_{1\frac{1}{2}}|^{\frac{1}{2}}}{(2\pi)^{\frac{np}{2}}} , \quad  \varTheta_{1\frac{1}{2}} = \sum_{j=1}^{p}  A_{jj} \otimes  B_{jj}, \\
    & c_{2} = \frac{| \varTheta_{2}|^{\frac{1}{2}}}{(2\pi)^{\frac{np}{2}}}, \quad   \varTheta_{2} = \sum_{i=1}^{n}\sum_{j=1}^{p} \gamma_{ij}  A_{i} \otimes  B_{j},\\
        & c_{3} = \frac{| \varTheta_{3}|^{\frac{1}{2}}}{(2\pi)^{\frac{np}{2}}} ,\quad   \varTheta_{3} =   \varPhi^{-1} \otimes  \varPsi^{-1},
    \label{prop 1}
\end{align}
where $\operatorname{vec}(\cdot)$ is the vectorisation of matrices and $\otimes$ is the Kronecker product. 
\label{prop: nested precision form}
\end{theorem}
\begin{proof}[Proof of Theorem \ref{prop: nested precision form}] Direct calculation.
\end{proof}
    
\begin{theorem}\label{prop: nested precision form 2}
     $X\sim T_{i} \Leftrightarrow \operatorname{vec}( X ') \sim N_{np}( 0;  \varTheta^{-1}_i)$, $i = 1, 1\frac{1}{2}, 2, 3$.
\end{theorem}

\begin{proof}[Proof of Theorem \ref{prop: nested precision form 2}]
    $\operatorname{tr} ( A X B{ X}') = {\operatorname{vec}( X')}^{\prime}( A \otimes  B) \operatorname{vec}( X' )$.
\end{proof}

\begin{corollary}\label{prop: nested precision form 3}
$T_{1}\supset T_{1\frac{1}{2}} \supset T_{2} \supset T_{3}$. 
\end{corollary}

\begin{theorem} 
The sample covariance $ X^{\prime} X$, where $ X \sim T_{i}, i=1, 1\frac{1}{2}, 2, 3$, is positive definite with probability one if and only if $n  > p -1 $.
    \label{lem: dykstral70}
\end{theorem}

\begin{proof}[Proof of Theorem \ref{lem: dykstral70}] 
    It suffices to show that ${X}^{\prime}{X}$ is non-singular with probability one if and only if $n> p -1$. Here, we will use ${x}_{1}, \dots, {x}_{p}$ to represent the column vectors of $X$ and similarly, ${\rm x}_{1}, \dots, {\rm x}_{n}$ the row vectors of $X$. 
	 \[\begin{aligned}	 	& {P} ({x}_{1}, \dots, {x}_{p} \text{ are linearly dependent})\\	 \leq &\sum_{i=1}^{p}{P} ({x}_{i} \text{ is a linear combination of }{x}_{1}, \dots, {x}_{i-1}, {x}_{i+1},{x}_{p})\\
     = & \sum_{i=1}^{p}{P} \left(\exists {b}_{1}, \dots, {b}_{i-1}, {b}_{i+1},{b}_{p} \in \mathbb R, {x}_{i}= \sum_{i^{\prime} \neq i}b_{i^{\prime}}x_{i^{\prime}}\right)\\
     \leq &  \sum_{j=1}^{p} \min_{i=1,2,\dots, n} {P} \left(\exists  {b}_{1}, \dots, {b}_{j-1}, {b}_{j+1},{b}_{p} \in \mathbb R, {x}_{ij}= \sum_{j^{\prime} \neq j}b_{j^{\prime}}x_{ij^{\prime}}\right)\\
     \leq &  p\cdot \min_{i=1,2,\dots, n} {P} \left(\exists b  \in \mathbb R^{k\times 1}, {\rm x}_{i} {b}^{\prime} = 0\right)= p \cdot 0 = 0\end{aligned} \]
    where in the last line, we use the fact that each ${\rm x}_{i}$ has independent components so that the $p$-dimensional normal distribution lies in the ($p-1$)-dimensional subspace with probability zero. 
\end{proof}

In practice, a problem occurs quite often that the number of parameters to be estimated is too large when we have no prior knowledge of the normal samples. Usually, the determination of the sample variances and covariances of $n$ samples in $k$ variables has, assuming normality for simplicity, $\frac{1}{2}nk(nk-1)$ correlation coefficients or regression coefficients to be estimated, approximately $O\big((nk)^2\big)$ unknown parameters. These notations $T_{1}$, $T_{1\frac{1}{2}}$, $T_{2}$, and $T_{3}$ promote an ad hoc approach to handle this dimensionality curse, and their stochastic representations may help understand the elliptical contoured distributions in the most general sense.

        \begin{table}[!h]
        \caption{Characterisations of left-spherical (LS), multivariate spherical (MS), vector spherical (VS) distributions and their extensions to elliptical distributions.}
    \centering
    \begin{tabular}{|c|c|c|c|c|}
        \hline
        Type & $-2\log(\text{M.G.F.})$  & D.O.F. & Class & M.G.F.\\
        \hline
        $T_{1}$ & $ \sum_{j=1}^{k}\sum_{j^{\prime}=1}^{k}
         t_{j}^{\prime}{ A}_{jj^{\prime}} t_{j^{\prime}}$ & $\frac{1}{2}n(n-1)k^2 + nk$ & LE & $\phi( t_{i}' A_{ij} t_{j})$\\
        \hline
        $T_{1\frac{1}{2}}$ &  $ \sum_{j=1}^{k} t_{j}^{\prime}{ A}_{jj} t_{j}$ & $\frac{1}{2}n(n+1)k$ & ME & $\phi( t_{i}' A_{ii} t_{i})$ \\
        \hline
        $T_{2}$ & $ \sum_{i=1}^{n}\sum_{j=1}^{k}\gamma_{ij}  t_{j}^{\prime}{ A}_{i} t_{j} $ & $ nk$ &  ME & $\phi( t_{i}' A_{ii} t_{i})$ \\
        \hline
        $T_{3}$ & $\sum_{i=1}^{n}\sum_{j=1}^{k}\alpha_{i}\beta_{j}  t_{j}^{\prime}{ A}_{i} t_{j}  $ & $n+k$ & VE & $\phi(\sum  t_{i}' A_{ii} t_{i})$\\
        \hline
    \end{tabular}\\
    \label{tab: matrix normal}
    where $ T = ( t_{1},\dots, t_{n})$ and $ X= ( x_{1},\dots, x_{n})$ are both $n \times p$ matrices, $ A_{ij}'=  A_{ji}$, and the M.G.F. is ${E}\exp (\sum t_{ri}x_{ri})$. These terminologies LE, ME, VE are slightly modified from \cite{1990Generalized}.
\end{table}

From Table \ref{tab: matrix normal}, we can see that the degrees of freedom are significantly reduced in the four matrix normal populations $T_{1}$, $T_{1\frac{1}{2}}$, $T_{2}$, and $T_{3}$, which is equivalent to the spectral decompositions of the four nested types of precision matrices described in Proposition \ref{prop: nested precision form 3}. 

\section{Inverse function method for generation}

The generation of matrix-variate normal distributions relies on the transformation of independent uniform samples into structured ones. This section outlines the transition from scalar white noise to the structured precision models $T_1, T_{1\frac{1}{2}}, T_2,$ and $T_3$ based on Box-Muller transformation and Proposition \ref{prop: stiefel}.

\subsection{Box-Muller transform}

The standard generation transform is a polar coordinate mapping proposed by \cite{box1958note}. Given two independent uniformly distributed samples $u_1, u_2 \sim U(0, 1)$, we define
\begin{eqnarray}
    z_0 = \sqrt{-2 \ln u_1} \cos(2\pi u_2), \quad 
    z_1 = \sqrt{-2 \ln u_1} \sin(2\pi u_2). \label{eq: BM transformation}
\end{eqnarray}
From the perspective of the inverse cumulative distribution function (CDF), let $X \sim N(0,1)$ with CDF $F(x)$. Then $F^{-1}(u)$ transforms $u \sim U(0,1)$ to a normal variable. The Box-Muller method bypasses the computationally expensive $F^{-1}$ by utilizing the fact that the joint distribution of two independent normal variables is rotationally invariant. Precisely, instead of attacking the one-dimensional inverse, the Box-Muller transformation considers two independent standard normal variables \(X,Y\). Their joint density is
\[
f(x,y)=\frac{1}{2\pi}e^{-(x^2+y^2)/2},
\]
which in polar coordinates \((R,\Theta)\) becomes
\[
f(r,\theta) r dr d\theta = \frac{1}{2\pi}e^{-r^2/2}r dr d\theta .
\]
Now \(R\) and \(\varTheta\) are not only independent, but their CDFs are easy to write down
\[
F_R(r)=1-e^{-r^2/2}\quad(r>0),\qquad
F_\varTheta(\theta)=\frac{\theta}{2\pi}\quad(0<\theta<2\pi).
\]
Inverting each gives
\[
R=\sqrt{-2\ln(1-U_1)},\qquad \Theta=2\pi U_2,
\]
and substituting back \(X=R\cos\varTheta,\;Y=R\sin\varTheta\), together with the fact that \(1-U_1\) and \(U_1\) share the same distribution, yields the classic form \eqref{eq: BM transformation}.

\subsection{Marsaglia-Bray method}

To avoid the evaluation of trigonometric functions, \cite{marsaglia1964convenient} preferred another polar method.
\begin{enumerate}
    \item Sample $v_1, v_2 \sim U(-1, 1)$ until $s = v_1^2 + v_2^2 < 1$ and $s \neq 0$.
    \item The transformed variables are
\end{enumerate}
\begin{equation}
    z_0 = v_1 \sqrt{\frac{-2 \ln s}{s}}, \quad z_1 = v_2 \sqrt{\frac{-2 \ln s}{s}}. \label{eq: Marsaglia method}
\end{equation}
This approach is particularly efficient when generating the large i.i.d. standard normal matrices $Z \in \mathbb{R}^{n \times k}$ required for matrix-variate models. Concretely, we use rejection sampling on \((-1,1)\times(-1,1)\) until a point \((v_1,v_2)\) falls inside the unit circle with \(s=v_1^2+v_2^2\in(0,1)\).  The point is then uniformly distributed over the disk, its angle \(\Theta\) is automatically uniform, and \(s\) follows a uniform distribution on \((0,1)\).  Setting \(R=\sqrt{-2\ln s}\) gives the radial part independent of \(\varTheta\). Hence, we have derived \eqref{eq: Marsaglia method} with no trigonometric calls. This method is markedly faster when generating the \(n\times k\) i.i.d. standard normal matrix \(Z\) that is the starting point for matrix-variate models.

\subsection{Gram-Schmidt orthogonalization}
This section addresses the generation of distributions where the precision matrix $\varTheta$ exhibits varying degrees of Kronecker or spectral separability.

\noindent 1. For \(T_3\), the precision is \(\varTheta_3 = \varPhi^{-1}\otimes\varPsi^{-1}\).  Compute Cholesky factors \(L,R\) such that \(\varPhi = LL'\), \(\varPsi = RR'\) and generate a standard normal matrix \(Z\) so that \(X = L Z R'\) follows the desired distribution.  

\noindent 2. For \(T_2\), \(\varTheta_2 = \sum_{i,j}\gamma_{ij}(a_i a_i')\otimes(b_j b_j')\) with orthonormal bases \(\{a_i\},\{b_j\}\).  Draw independent \(z_{ij}\sim_{i.i.d.} N(0,1)\) and set \(X = \sum_{i,j}\gamma_{ij}^{-1/2}z_{ij}a_i b_j'\).  

\noindent 3. For \(T_{1\frac{1}{2}}\), columns are uncoupled in the basis \(\{b_j\}\) but have precisions \(A_{jj}\).  For each \(j\), factor \(A_{jj}^{-1} = L_j L_j'\), take \(z^{(j)}\sim_{i.i.d.} N(0,I_n)\), and form \(y^{(j)} = L_j z^{(j)}\).  Stack the resulting vectors and rotate by \(B=(b_1,\dots,b_k)\) to obtain \(X = (y^{(1)},\dots,y^{(k)})B'\).

\noindent 4. The most general case has precision \(\Theta_1 = \sum_{j,j'} A_{jj'}\otimes(b_j b_{j'}')\) with a block matrix \(A=(A_{jj'})\).  Perform a block Cholesky decomposition \(A = LL'\), where \(L = (L_{mj})\) is block lower-triangular.  Generate \(z^{(j)}\sim_{i.i.d.} N(0,I_n)\) and compute the forward recursion equation
\[
y^{(j)} = L_{jj}^{-1}\Bigl(z^{(j)} - \sum_{m=1}^{j-1} L_{jm}\,y^{(m)}\Bigr), \qquad j=1,\dots,k.
\]
Finally, rotate to the basis \(\{b_j\}\) by setting \(X = (y^{(1)},\dots,y^{(k)})B'\). 

\subsection{Classification of flag varieties}

The indicator map of these four matrix normal distributions $\Gamma: {\rm Sym}_{nk}^+ \to \mathcal{P}_{0}(M)$ is defined as follows, where ${\rm Sym}_{nk}^+$ represents the $nk\times nk$ real symmetric positive definite matrices, namely the precision matrices of ${\rm vec}(X')$. From Theorem \ref{lem: dykstral70}, this map is well-defined on the configuration spaces of precision matrices $\mathcal{C}_i, i =1,1\frac{1}{2},2,3$, corresponding to $T_1,T_{1\frac{1}{2}},T_2,T_3$. The following theorem classifies the injectivity of this map. From the above generation algorithms, we have established a one-to-one correspondence of the structure of precisions $\varTheta$ onto matrix normal distributions $X$. Thus, we have the following classification theorem.

\begin{theorem} The following statements are true.
\begin{enumerate}
    \item The normal map $\Gamma$ is injective if and only if $M$ contains an analytic bundle with base ${\rm Flag}(1,2,\dots,k;\mathbb R^{n+k})$ and fibre $O(1)^k$.
    \item The normal map $\Gamma|_{\{2,3\}}$ is injective if and only if $M$ contains an analytic bundle with base ${\rm Flag}(\mathbb R^{k})$ and fibre $O(1)^k$.
\end{enumerate}
\label{thm: flag}
\end{theorem}
\begin{question}
    Should we characterise more classes up to isomorphism between the double flag and the single flag?
\end{question}

\section{Real gamma distribution}

To calculate the singular values of such normal map, it requires us to derive the distribution of the quadratic form $X'X$ in a normal matrix $X$. By this, standard multi-varaite statistical analysis techniques like the determinant, trace, largest and smallest roots and their ratio could to performed. The simplest situation is the Wishart distribution, according to $T_3$ with $\varPhi = I_n$. As a natural extension of independent identically distributed normal vectors, this distribution was further generalised by \cite{anderson1946wishart}, \cite{James1955TheNW}, and \cite{herz1955} etc. \cite{james1960distribution,james1961distribution,james1961zonal,james1968calculation} use the hypergeometric series in terms of zonal polynomials to derive the exact distribution of such a distribution with non-central means, removing many redundant proofs, such as in \cite{Kotz1967ab}. In the following of this section, we are going to derive the distribution of the quadratic form of normal matrices $T_1$.

Evidently, a normal matrix has a quadratic form representation \eqref{eq: illustration of triangular ararry heter}. This joint distribution of the homogeneous quadratic forms, that is, $A_{ij} = A$, a fixed precision matrix was first obtained by \cite{khatri1966} with the hypergeometric function ${}_0F_0$ of two matrix arguments.  In his work, the exact distribution is not expressed in the series of zonal polynomials such as \cite{Kotz1967ab}, but his method involves nuisance parameters $q_{ij} >0$. This expression is difficult in practice since the density is the real gamma distribution in the central Wishart case, where the hypergeometric function ${}_0F_0$ does not appear explicitly. In addition, hypergeometric functions of two matrix arguments often occur in the characteristic function and distribution of characteristic roots for the central case, instead of the density as noted by \cite{james1964distribution}. Thus, after nearly thirty years, the heterogeneous quadratic forms were reconsidered in \cite{Mathai1992} by calculating the cumulants only for $k=2$. However, recent advances in normal quadratic forms from aspects of combinatorial connections, such as \cite{hanlon1992combinatorial}, \cite{gouldenandjackson1995}, and \cite{vassilieva2015moments} mostly concern the calculation of moments, instead of exact density. It lies in another core subject of this article, the determination of the central density of these $\frac{1}{2}k(k+1)$ heterogeneous quadratic forms for general $p>2$ using no hypergeometric function and unnecessary constant $q_{ij}$. To begin with a reasonable foundation, these matrices must satisfy $A_{ij}' = A_{ji}$ and the diagonal entries in $A_{ij}(i\neq j)$ are equal to zero.  As a natural generalisation of the matrix normal distribution in equation \eqref{eq: dawid matrix normal distribution} to \eqref{eq: illustration of triangular ararry heter}, the advantage is an organised representation of the elliptical contoured distribution, such as Theorem 3.6.7 in \cite{1990Generalized}, which may promote a better understanding of general symmetric distributions. 
\subsection{Probability density function}

\begin{theorem} \label{thm: main theorem}  Suppose $X$ is an $n\times k$  real normal matrix $T_1$ with $B = ( b_{1}, b_{2},\dots, b_k)$,
Let $ M$ be a fixed real $n\times k$ matrix and $ S =  (X +  M)'(X +  M)$. Assume that the diagonals in the $k\times k$ real matrix $ U = (u_{ij})$ with $u_{ij} = \operatorname{tr}( A_{ij})$ are all positive, that is, $u_{ii}>0 \, (i=1,2,\dots,k)$. Then the probability density distribution of $ S$ only depends on $ T 
= (t_{ij})$ with 
    $t_{ij} = \operatorname{tr}( B_{ij}  S)$ when $n > k - 1$, and is 
        \begin{eqnarray}
         \frac{| \varTheta_{1}|^{\frac{1}{2}}}{2^{\frac{nk}{2}}\varGamma_{k}(\frac{n}{2})} \operatorname{etr} \left(-\frac{1}{2}  U T \right) | T|^{\frac{n-k-1}{2}}          \text{ when } {M}=0; \\
                \times \operatorname{etr} \left(-\frac{1}{2}  \varOmega\right){}_{0}F_{1}\left(\frac{n}{2};\frac{1}{4} \varDelta  T\right) \text{ when } {M}\neq0;
        \end{eqnarray}
        where $\varOmega = \sum_{i,j=1}^{k} B_{ij}  M^{\prime}   A_{ij}  M$ and $ \varDelta = \sum_{i,j,k,l=1}^{k} B'_{ij} M'A'_{ij}A_{kl} M B_{kl}$.
\end{theorem}

\begin{corollary}\label{lem: any p}
The distribution of the quadratic form $X'X$ in \eqref{eq: dawid matrix normal distribution} has the density in the central case $M = 0$ for any $n > p-1$,
\begin{equation}
    \begin{aligned}
    h_{n,k}( S) = \frac{ \operatorname{etr}\left( -q^{-1} \varPsi^{-1} S\right)| S|^{\frac{n-k-1}{2}} }{2^{\frac{nk}{2}}\varGamma_{k}(\frac{n}{2})| \varPhi|^{\frac{k}{2}}| \varPsi|^{\frac{n}{2}}}{_{0}}F_{0}\left( I - \frac{1}{2}q \varPhi^{-1}, q^{-1} \varPsi^{-1} S\right), 
\end{aligned}\label{eq: product moment distribution in q}
\end{equation}
where $q$ is an arbitrary positive constant.
\end{corollary}

\begin{remark}\label{rem: reduction}
    In fact, the central result in Theorem \ref{thm: main theorem} may be reduced to equation \eqref{eq: product moment distribution in q} by integrating the $\frac{1}{2}k(k+1)$ quadratic forms $y_{i} A_{ij} y_{j}$($i\leq j$) separately in \eqref{eq: T1}, in terms of $ A_{ij} =  A_{ji}'$. Conversely, if Theorem \ref{thm: main theorem} is true, then equation \eqref{eq: product moment distribution in q} holds for any $n > k-1$.
\end{remark}

    The proofs are organised as follows. We are going to prove, for the first instance, the non-central part of Theorem \ref{thm: main theorem} based on the equivalence of the central part and Corollary \ref{lem: any p}, and thereafter, the equivalence. Since Corollary \ref{lem: any p} is a known result with the help of Proposition \ref{lem: dykstral70}, the proof of Theorem \ref{thm: main theorem} is completed.

\begin{proof}[Proof of Theorem \ref{thm: main theorem}]
    Suppose the central part holds for $n > k-1$. Decomposing $ X =  H  Z$ where $ H' H =  I_p$.
    Let $K = (H,H_{\perp}) \in O(n)$. From Lemma \ref{prop: herz55}, we could rewrite the density in $T_1$ as 
    \[\begin{aligned}
        \operatorname{etr} \left(-\frac{1}{2}  \varOmega\right)\int_{O_n}\operatorname{etr}\left(-\frac{1}{2} \sum_{i,j=1}^{k}( A_{ij} M B_{ij})Z'H' \right) (d K) \\
        = \operatorname{etr} \left(-\frac{1}{2}  \varOmega\right){}_{0}F_{1}\left(\frac{n}{2};\frac{1}{4}\sum_{i,j,k,l=1}^{k}( B'_{ij} M'A'_{ij}A_{kl} M B_{kl}) Z'Z\right).    \end{aligned}\]
    Thus, if the central part is true, Theorem \ref{thm: main theorem} is consequently proved by applying Lemma \ref{lem: james55}. 
    
    From Lemma \ref{lem: Khatri66} on the quadratic forms in normal vectors, we can reduce the central part of Theorem \ref{thm: main theorem} to the $\frac{1}{2}k(k+1)$ independent quadratic forms by introducing $ Y =  X B = ( y_{1}, y_{2}\dots,  y_{k})$
    \begin{equation*}
    \begin{aligned}
        {y}_1'{A}_{11}{y}_1, {y}_1'{A}_{12}{y}_2, \dots, {y}_1'{A}_{1k}{y}_k,\\
        {y}_2'{A}_{22}{y}_2, \dots, {y}_2'{A}_{2k}{y}_k, \\
        \vdots\qquad \\
        {y}_k'{A}_{kk}{y}_k,
    \end{aligned}
\end{equation*}
In fact, we have for each ${y}_i'{A}_{ij}{y}_j$ the contribution to the probability density function
\[\begin{aligned}
    {_{0}}F_{0}\left( I - \frac{q_{ij}}{2} A_{ij}, q_{ij}^{-1} y_{i} y_{j}^{\prime}\right) = {_{0}}F_{0}\left( I - \frac{q_{ij}}{2} A_{ij}, q_{ij}^{-1} y_{j}^{\prime} y_{i}\right).\end{aligned}\]
However, we may find $ y_{j}^{\prime} y_{i} = t_{ij} = \operatorname{tr} ( B_{ij} S)$ so that from these properties of hypergeometric functions
\[\begin{aligned}
    {}_{0}F_{0} ( X,c Y) = {}_{0}F_{0} (c X, Y)\\
    {}_{0}F_{0} ( X,  I) = {}_{0}F_{0} ( X) = \operatorname{etr}( X), 
\end{aligned}\]
the terms concerning $q_{ij}$ cancel out. This yields the desired form.
\end{proof}

\begin{proof}[Proof of Corollary \ref{lem: any p}]
    Let $ A_{ij} = \beta_{j}\delta_{ij}   \varPhi^{-1}$ such that $ \varPsi^{-1} = \sum_{j=1}^{k}\beta_j b_{j} b_{j}^{\prime}$ where $\beta_{j},  b_{j}$ are the latent roots and vectors of $ \varPsi^{-1}$. 
Then $u_{ij} = \beta_{j}\delta_{ij}\operatorname{tr} ( \varPhi^{-1})$ and 
\[\begin{aligned}
    \operatorname{etr}\left(-\frac{1}{2} U  T \right) = \operatorname{etr}\left(- \frac{1}{2}\operatorname{tr} ( \varPhi^{-1})\sum_{j=1}^{k}\beta_{j}t_{jj}\right)     = \operatorname{etr}\left(- \frac{1}{2} \varPhi^{-1}\operatorname{tr} ( \varPsi^{-1}  S)\right) \\
    = \operatorname{etr} \left(- q^{-1} \varPsi^{-1}  S\right){}_{0}F_{0} \left( I - \frac{q}{2} \varPhi^{-1}, q^{-1} \varPsi^{-1}  S\right),
\end{aligned}\]
where, in the second equality, we used the definition $ T =  B  S  B'$.
\end{proof}

\subsection{Moment generating function}

\begin{theorem} \label{thm: main theorem 2}  Assume as Theorem \ref{thm: main theorem}. The moment generating function of  $ S = (s_{ij})$ is 
        \begin{eqnarray}
        {E}\exp \left(\sum_{i\leq j}\gamma_{ij}s_{ij}\right) = | \varTheta_{1}|^{\frac{1}{2}} | U|^{-\frac{n}{2}}{_{1}}F_{0} \left(\frac{n}{2};  W\right) \text{ when } {M}=0; \label{eq: central mgf}\\
            \times \operatorname{etr} \left(-\frac{1}{2}  \varOmega\right) \operatorname{etr} \left(\frac{1}{2}  \varDelta   U^{-1}( I -  W)^{-1}\right) \text{ when } {M}\neq0;\label{eq: non-central mgf} 
        \end{eqnarray}
                    where $ W =   U^{-\frac{1}{2}}  B R B^{\prime} U^{-\frac{1}{2}} $, $2 R =  \varGamma +  I$, and $ \varGamma = (\gamma_{ij})$ symmetric. 
\end{theorem}
\begin{proof}[Proof of Theorem \ref{thm: main theorem 2}]
    Suppose Theorem \ref{thm: main theorem} holds for $n > k-1$. In order to prove Theorem \ref{thm: main theorem 2} when $n > k-1$, after the transformation $ S \mapsto  T =  B' S B$, it becomes
\[\begin{aligned}
    \frac{| \varTheta_{1}|^{\frac{1}{2}}}{2^{\frac{nk}{2}}\varGamma_{k}(\frac{n}{2})}   \int_{ T >  0} \operatorname{etr} \left(-\frac{1}{2} ( U -  B R  B') T\right) | T|^{\frac{n-p-1}{2}} \\
    \times   \operatorname{etr} \left(-\frac{1}{2}  \varOmega\right)    {}_{0}F_{1}\left(\frac{n}{2};\frac{1}{4} \varDelta  T\right)(d  T).
\end{aligned}\]
By substituting with $Q =  \varDelta  T$ and $(d Q) = | \varDelta|^{\frac{k+1}{2}}(d T)$, we can calculate this integral explicitly using the Gamma integral. 
\end{proof}




\begin{corollary}\label{lem: mgf} The moment generating function of $h_{n,k}( S)$ in \eqref{eq: product moment distribution in q} is
    \[| \varPhi|^{-\frac{k}{2}} | W|^{-\frac{n}{2}} {_{1}}F_{0} \left(\frac{n}{2};q I - \frac{1}{2} \varPhi^{-1},  W^{-1}\right),\]
     where $ W =  I - q  \varPsi^{\frac{1}{2}}  R \varPsi^{\frac{1}{2}}$, $2 R =  \varGamma +  I$ and $ \varGamma = (\gamma_{ij})$ symmetric. 
\end{corollary}
A proof can be found in \cite{khatri1966}. Thereby, Theorem \ref{thm: main theorem 2} generalises the classical results of normal quadratic forms in a compact form with $${}_1F_{0}\left(a;W\right) = |I - W|^{-a},$$
similar to the Wishart moment generating function.
\begin{question}\label{ques: mgf} 
    Can one show that the equation \eqref{eq: central mgf} determines a moment generating function if and only if $n\in\{0,1,2,\dots,k-1\}\cup (k-1,\infty)$ and \eqref{eq: non-central mgf} determines a moment generating function if and only if, additionally, $n \geq \max\{\operatorname{rank}(\varOmega),\operatorname{rank}(\varDelta)\}$ when $n < k-1$? 
\end{question}

This answer to Question \ref{ques: mgf} (if true) reduces to the conjecture of \cite{peddada1991} only when $A_{jj'} = \beta_{j}I\delta_{jj'}$ or $\varSigma = I \otimes \varPsi$, which has recently received proofs by \cite{letac2018laplace} using recursive methods and \cite{mayerhofer2019wishart} with continuation arguments. Our proposal is a similar answer for $T_1$ and others, so it remains novel.

\subsection{Distribution of characteristic roots}
\begin{theorem} \label{thm: main theorem 3} Assume the same as Theorem \ref{thm: main theorem}. The joint distribution of characteristic roots $l_{1}, l_{2}, \dots, l_{k}$ of  $ S$ is 
        \[\begin{aligned}
        			\frac{\pi^{\frac{k^{2}}{2}}| \varTheta_{1}|^{\frac{1}{2}}}{2^{\frac{nk}{2}}\varGamma_{k}(\frac{n}{2})\varGamma_{k}(\frac{k}{2})}
                    \prod_{i<j}^{k} (l_{i} - l_{j}) 
                    \prod_{i=1}^{k}l_{i}^{\frac{n-k-1}{2}}
                    {}_0F_{0}\left(-\frac{1}{2}  U,  L\right),  \text{ when } {M}=0;\\
        \times \operatorname{etr} \left(-\frac{1}{2} \Omega\right){}_{0}F_{1}\left(\frac{n}{2};\frac{1}{4} \Delta, L\right) \text{ when } {M}\neq0; 
        \end{aligned}\]
where $ L = \operatorname{diag}(l_i)$, $l_{1} > l_{2} > \dots > l_{k}>0$; elsewhere zero.
\end{theorem}

    \begin{proof}[Proof of Theorem \ref{thm: main theorem 3}]
     First, the central part is a consequence of the gamma integral. From \cite[Theorem 3.2.17]{muirhead1982aspects}, decomposing $ X =  H  Z$ where $H' H =  I_k$ with $K = (H,H_{\perp}) \in O_n$, we have by Proposition \ref{lem: dykstral70} and Theorem \ref{thm: main theorem},
    \[\begin{aligned}
        \operatorname{etr} \left(-\frac{1}{2}  \Omega\right)\int_{O_p}\operatorname{etr}\left(-\frac{1}{2}\sum_{i,j=1}^{k}( A_{ij} MB_{ij}) Z' H'\right)(d K) \\
        = \operatorname{etr} \left(-\frac{1}{2}  \Omega\right){}_{0}F_{1}\left(\frac{n}{2};\frac{1}{4}\sum_{i,j,k,l=1}^{k}( B_{ij}' M' A_{ij}' A_{kl} M B_{kl}) Z' Z\right).
    \end{aligned}\]
    The proof for the non-central part of Theorem \ref{thm: main theorem 3} is consequently done by applying Lemma \ref{lem: james55}. 
    \end{proof}

\begin{corollary}\label{lem: latent} The joint distribution of characteristic roots $l_{1}, l_{2}, \dots, l_{k}$ of $S$ in \eqref{eq: product moment distribution in q} is
    \begin{equation*}
        \begin{aligned}
      \frac{\pi^{\frac{k^{2}}{2}}}{2^{\frac{nk}{2}}\varGamma_{k}(\frac{n}{2})\varGamma_{k}(\frac{k}{2})| \varPhi|^{\frac{k}{2}}| \varPsi|^{\frac{n}{2}}} \prod_{i=1}^{k}(l_i)^{\frac{n-k-1}{2}}\prod_{\substack{i<j}}^{k} (l_{i} - l_{j})
    {}_0F_{0}\left(-\frac{1}{2} \varPhi^{-1},  \varPsi^{-1} L\right)
        \end{aligned}
    \end{equation*}
    where $ L = \operatorname{diag}(l_{i})$, $l_{1} > l_{2} > \dots > l_{k}>0$; elsewhere zero. 
\end{corollary}

\begin{proof}[Proof of Corollary \ref{lem: latent}]
     We introduce $ K = (H_1,  H_{2}) \in O_n$, where $ H_{1}$ is an $n \times k$ matrix such that $ X  H =  H_{1}  Q$ and $H_1'H_1 = I_k$. From Proposition \ref{prop: herz55}, we have
\begin{equation}
    \begin{aligned}
       p( L)    \propto \int_{O_n} (dK) \int_{ Q'  Q = L } \operatorname{etr} \left( -\frac{1}{2} \varPhi^{-1} H_1  Q \varPsi^{-1} Q^{\prime} H_1' \right) (d Q)\\
        \propto \int_{O_n}  {}_0F_{0}\left(-\frac{1}{2} H_1'  \varPhi^{-1}  H_1,  \varPsi^{-1} L\right)(dK).
    \end{aligned}
    \label{eq: latent roots in hypergeometric integral}
\end{equation}
Then Lemma \ref{lem: latent} follows from this equation \eqref{eq: latent roots in hypergeometric integral} by multiplying a leading coefficient and the Vandermonde determinant $\prod_{\substack{i<j}}^{p} (l_{i} - l_{j})$.
\end{proof}

To illustrate how this is equivalent to the result of \cite{james1960distribution} when $ A =  I$, we should integrate \eqref{eq: dawid matrix normal distribution} with respect to $X$ and $H$ such that $ X^{\prime} X= H  L  H'$. In fact, by introducing $ Z =  X  H$, we have from Lemma \ref{lem: Khatri66} that
\begin{equation*}   \begin{aligned}
    p( L) \propto \int_{O(p)} (d H) \int_{ X'  X = H L H'} \operatorname{etr} \left( -\frac{1}{2}  \varPhi^{-1}{X} \varPsi^{-1}{X}^{\prime} \right)(d X)\\
   \propto \int_{O(p)} (d H) \int_{ Z'  Z = L } \operatorname{etr} \left( -\frac{1}{2}  \varPhi^{-1}{Z} H \varPsi^{-1} H'{Z}^{\prime} \right) (d Z)\\
    \propto | L|^{\frac{n-k-1}{2}}\int_{O(p)}   {}_0F_0\left( -\frac{1}{2}  \varPhi^{-1}, H \varPsi^{-1} H' L\right) (d H).\end{aligned}\end{equation*}
Comparing the leading coefficient, this reduces to the classical result of James on the joint distribution of latent roots of \eqref{eq: central Wishart latent roots} when $ \varPhi =  I$,
\begin{eqnarray}
          \frac{\pi^{\frac{k^{2}}{2}}}{2^{\frac{nk}{2}}\varGamma_{k}(\frac{n}{2})\varGamma_{k}(\frac{k}{2})| \varPsi|^{\frac{n}{2}}} \prod_{i=1}^{k}(l_i)^{\frac{n-k-1}{2}}\prod_{\substack{i<j}}^{k} (l_{i} - l_{j}) {}_0F_{0}\left(-\frac{1}{2} \varPsi^{-1},  L\right), \label{eq: central Wishart latent roots}
\end{eqnarray}
where $ L = \operatorname{diag}(l_{1},l_{2},\dots,l_{k}), l_{1}>l_{2}>\dots > l_{k}$; elsewhere zero. 

Thus, when $\varPhi =  I$, Corollary \ref{lem: latent} reduces to the classical result of James on the distribution of latent roots of the central Wishart distribution.




\begin{example}[Wishart distribution when $k=3$] Let the frequency distribution of the population sampled be
\begin{equation}
    \begin{aligned}
p( X) = \frac{\prod_{r=1}^{n}| B_{r}|^{\frac{1}{2}}}{(2\pi)^{\frac{3n}{2}}} \exp \left(-\frac{1}{2}\sum_{r=1}^{n}\textrm{x}_r^{\prime}  B_{r} \textrm{x}_{r}\right).
\end{aligned}
\label{eq: Wishart density}
\end{equation}
where $ X = [\textrm{x}_1;\textrm{x}_2;\dots;\textrm{x}_n]$ is the population with row vectors independent normally distributed and $ B_{r}$ is the inverse of the $3 \times 3$ covariance matrix of $\textrm{x}_r = (x_{r1},x_{r2},x_{r3})$.

The following statistics are now to be calculated from the sample
$$\begin{aligned}
n \bar{x}_1 = \sum_{r = 1}^nx_{r1}, \,
n \bar{x}_2 =\sum_{r=1}^nx_{r2}, \
n \bar{x}_3 =\sum_{r=1}^nx_{r3}, \\
n s_1^2=\sum_{r=1}^n(x_{r1}-\bar{x}_{1})^2,\, 
n s_2^2=\sum_{r=1}^n(x_{r2}-\bar{x}_{2})^2, \,
n s_3^2=\sum_{r=1}^n(x_{r3}-\bar{x}_{3})^2, \,\\
n r_{ij} s_i s_j=\sum_{r=1}^n(x_{ri}-\bar{x}_{i})(x_{rj}-\bar{x}_{j}), \, (i,j=1,2,3).
\end{aligned}$$
In order to transform the element of volume to $\bar{x}_1,\bar{x}_2,\bar{x}_3,s_{1},s_{2},s_3$, according to known results in \cite{Tumura1965THEDO}, we have the contribution to the transformed element of volume by introducing two axillary angles $\theta_{11}$, $\theta_{12}$, and $\theta_{22}$, representing the generators for the $3\times 3$ orthogonal group
\begin{equation}    \begin{aligned}   dp=    \frac{n^{\frac{3(n-2)}{2}}\prod_{r=1}^{n}| B_{r}|^{\frac{1}{2}}}{2^{\frac{3(n-1)}{2}}\pi^{\frac{3}{2}}\Gamma(\frac{n-1}{2})\Gamma(\frac{n-2}{2})\Gamma(\frac{n-3}{2})} \exp \left(-\frac{1}{2}\sum_{r=1}^{n}\textrm{x}_r^{\prime}  B_{r} \textrm{x}_{r}\right) s_{1}^{n-2}s_{2}^{n-2}s_{3}^{n-2}\\\times \sin^{n-3}\theta_{11}\sin^{n-3}\theta_{12}\sin^{n-4}\theta_{22} d\bar{x}_1d\bar{x}_2d\bar{x}_3ds_{1}ds_{2}ds_{3}d\theta_{11} d\theta_{12}d\theta_{22}.  \end{aligned}\label{eq: multivariate spherical normal in theta12 and phi}\end{equation}


By integrating \eqref{eq: multivariate spherical normal in theta12 and phi} with respect to $\theta_{11},\theta_{12},$ and $\theta_{22}$, we obtain the simultaneous distribution of the three variances and the three product moment coefficients when $ B_{1} =  B_2 = \dots =  B$ the symmetric expression
\begin{equation}    \begin{aligned}& A= {| B_{11}|}/{| B|}, \quad B= {| B_{22}|}/{| B|}, \quad C={| B_{33}|}/{| B|}, \\ &F= {| B_{12}|}/{| B|}, \quad G=  {| B_{13}|}/{| B|}, \quad H= {| B_{23}|}/{| B|}, \\&
a=s_{1}^2, \quad b=s_{2}^2,\quad c=s_{3}^2, \quad f = s_{1}s_{2}\cos \theta_{11}, \quad g = s_{1}s_{3}\cos \theta_{12}, \\&
h = s_{2}s_{3}(\sin\theta_{11}\sin \theta_{12} \cos \theta_{22} + \cos\theta_{11}\cos \theta_{12}),\\& d p=\frac{1}{\pi^{\frac{3}{2}} \Gamma (\frac{n-1}{2}) \Gamma(\frac{n-2}{2}) \Gamma(\frac{n-3}{2})}\cdot \left|\begin{array}{cccc}A & F & G \\ F & B & H \\G & H & C\end{array}\right|^{\frac{n-1}{2}} \cdot \left|\begin{array}{ccc}a & f & g \\f & b & h \\g & h & c\end{array}\right|^{\frac{n-5}{2}} \\&\qquad\qquad\qquad\qquad\qquad\times  e^{-A a-B b-C c-2 H h-2 G g-8 F f} d a d b d c d f d g d h,\end{aligned}\label{eq: Wishart}\end{equation}
where $ B_{ij}$ is the ($i,j$)-minor of $ B$, multiplying a common constant $({n}/{2})$. 


Wishart first obtained the formula \eqref{eq: Wishart} based on results from many famous authors such as Fisher, K. Pearson, and Romanovsky before 1925. The cited articles can also be found therein. Here, we utilize the techniques from Tumura in the decomposition of orthogonal matrices to derive the same result.

For other values of $p$, the product moment distribution is easily derived by the above method when the rows are independent identically distributed samples from a multivariate normal population. These results are well-summarized in books such as \cite{anderson1958introduction}, \cite{srivastava2009introduction}, and \cite{muirhead1982aspects}, \cite{eaton2007multivariate}, \cite{gupta2018matrix}, \cite{mathai2022mul}.
\end{example}    

\subsection{Extreme eigenvalue distributions}

In this section, we are going to derive the extreme eigenvalue distribution of Theorem \ref{thm: main theorem}. The largest eigenvalue of the Wishart distribution is obtained by \cite{roy1953heuristic}, \cite{sugiyama1967distribution,sugiyama1972distributions}, \cite{johnstone2001distribution,johnstone2008multivariate}, \cite{johnstone2017roy}, and \cite{kan2019densities}. Here, our first theorem implements them. However, the smallest eigenvalue, as noted by \cite{constantine1963some}, is usually difficult to evaluate due to this asymmetry phenomenon $P(S>R) \neq 1 - P(S<R)$. Our second theorem that will be given is about the smallest eigenvalue distribution based on Theorem \ref{thm: main theorem} and the \cite{khatri1972exact}'s truncated series of zonal polynomials. We are not to give the definition of zonal polynomials here, while the interested readers could find them in \cite{hua1963harmonic}, \cite{muirhead1982aspects}, \cite{macdonald1998symmetric}, \cite{Vilenkin2010RepresentationOL}.

        \begin{theorem}
	If $R$ is an $k \times k$ real symmetric positive definite matrix, then the probability that $R - S$ is positive definite is
	\[\begin{aligned}
		P(S < R) = \frac{| R|^\frac{n}{2} }{{B}_{k}\left(\frac{n}{2},\frac{k+1}{2}\right)}
		\operatorname{etr} \left(-\frac{1}{2} \varOmega\right) {}_{0}F_{1} \left(n + \frac{k+1}{2}; -  U  R, \varDelta  R\right)\end{aligned}.\]
\end{theorem}
\begin{proof}
	By the Constantine-Davis lemma on the incomplete Gamma integral, 
	\[\begin{aligned}
		P(S < R) = & 	\frac{| \varTheta_{1}|^{\frac{1}{2}}}{2^{\frac{nk}{2}}\varGamma_{k}(\frac{n}{2})} \int_{0}^{BRB'}  \operatorname{etr} \left(-\frac{1}{2} UT \right) | T|^{\frac{n-k-1}{2}} \\
		& \times \operatorname{etr} \left(-\frac{1}{2} \varOmega\right){}_{0}F_{1}\left(\frac{n}{2};\frac{1}{4} \varDelta T\right) (dT)\\
		= & \frac{| R|^\frac{n}{2} }{{B}_{p}\left(\frac{n}{2},\frac{k+1}{2}\right)}
		\operatorname{etr} \left(-\frac{1}{2} \varOmega\right) {}_{0}F_{1} \left(n + \frac{k+1}{2}; -  U  R, \varDelta  R\right).\qedhere
	\end{aligned}
	\]
\end{proof}
\begin{corollary} The distribution of largest root $l_{1}$ of $S$ is 
	\[\begin{aligned} 
		P(l_{1} < x) =  \frac{|\frac{1}{2}x|^{nk/2}}{B_p\left(\frac{n}{2},\frac{k+1}{2}\right)}&\operatorname{etr} \left(-\frac{1}{2}\varOmega\right){}_{0}F_{1} \left(n + \frac{k+1}{2}; -  xU, x\varDelta\right).\end{aligned}\]
\end{corollary}
\begin{proof}
	Let $R = xI_{p}$ so that $l_{1} < x$ equivalent to $S < xI_{p}$.
\end{proof}

\begin{theorem}If $R$ is an $k \times k$ positive definite matrix then the probability that $S - R$ is positive definite for the central case is
        \[\begin{aligned}
            {P}(S > R) = & \frac{| \varTheta_{1}|^{\frac{1}{2}}|R|^{\frac{n}{2}}}{2^{\frac{nk}{2}}\varGamma_{k}(\frac{n}{2})}\operatorname{etr}\left(-\frac{1}{2} UBRB'\right) \sum_{|\kappa|=1}^{kr} {\sum_{\kappa}}^*\frac{C_{\kappa}\left(\frac{1}{2} UBRB'\right) }{|\kappa|!}.\end{aligned}\]
            where $C_{\kappa}$ denotes the zonal polynomial and $\sum_{\kappa \vdash k}^{*}$ means summation over those partitions $\kappa = (\kappa_{1},\dots,\kappa_{p})$ with $\kappa_{1} \leq \frac{1}{2}(n-k-1)$.
        \end{theorem}
        \begin{proof}Putting $S = R^{1/2}(I_{p} + W)R^{1/2}$ with $(dS) = |R|^{(k+1)/2}(dW)$ we get 
          \[\begin{aligned}
                P(S > \Omega) = & \frac{| \varTheta_{1}|^{\frac{1}{2}}|R|^{\frac{n}{2}}}{2^{\frac{nk}{2}}\varGamma_{p}(\frac{n}{2})} \int_{W>0}  \operatorname{etr} \left(-\frac{1}{2} UBR^{1/2}WR^{1/2}B' \right) |I_{k} + W|^{\frac{n-k-1}{2}} \\
		        \times &\operatorname{etr}\left(-\frac{1}{2} (UBRB'+ \varOmega)\right){}_{0}F_{1}\left(\frac{n}{2};\frac{1}{4} \varDelta BR^{1/2}(I_{k} + W)R^{1/2}B'\right) (dW).
            \end{aligned}
            \]
            Using the fact that with $r = \frac{1}{2}(n-k-1)$,
            \[|I_{k} + W^{-1}|^{r} = {}_{1}F_{0}(-r;-W^{-1}) = \sum_{k=0}^{kr}{\sum_{\kappa \vdash k}}^{*}\frac{(-r)^{\kappa} C_{\kappa}(-W^{-1})}{k!},\]
            and substitution $W = V^{*}V$ where ${V}$ is upper-triangular with positive diagonal elements gives the central result for $\varOmega = \varDelta = 0$,
         \[\begin{aligned}
                {P}(S > R) = & \frac{| \varTheta_{1}|^{\frac{1}{2}}|R|^{\frac{n}{2}}}{2^{\frac{nk}{2}}\varGamma_{k}(\frac{n}{2})}\operatorname{etr}\left(-\frac{1}{2} UBRB'\right) \sum_{|\kappa|=1}^{kr} {\sum_{\kappa}}^*\frac{C_{\kappa}\left({\frac{1}{2} UBRB'}\right) }{|\kappa|!}.\qedhere
            \end{aligned}
            \]
        \end{proof}
        
\begin{corollary} The distribution of smallest root $l_{k}$ of $S$ is 
	\[\begin{aligned} 
		P(l_{k} & < x) =  1 - P(l_{k} > x)  \\
        & = 1 - \frac{| \varTheta_{1}|^{\frac{1}{2}}|\frac{1}{2}x|^{\frac{nk}{2}}}{\varGamma_{k}(\frac{n}{2})}\operatorname{etr}\left(-\frac{1}{2} UBRB'\right) \sum_{|\kappa|=1}^{kr} {\sum_{\kappa}}^*\frac{C_{\kappa}\left(\frac{1}{2} UBRB'\right) }{|\kappa|!}.\end{aligned}\]
\end{corollary}
\begin{proof}
	Let $\Omega = xI_{p}$ so that $l_{k} > x$ equivalent to $S > xI_{p}$. Thus, from the continuity of the distribution, $P(l_{k} < x) =  1 - P(l_{k} > x)$.
\end{proof}
The non-central distribution of the smallest latent root is a bit long and tedious, so we omit it.



\section{Conclusion}

In this paper, only the real gamma distribution has been considered, so extensions to the complex case remain future work. Also, for $T_{1\frac{1}{2}}$, $T_{2}$, and $T_{3}$, results can be established in parallel. In terms of similar results of flag varieties like Theorem \ref{thm: flag}, the classification remains rather interesting and far from complete. Here, this article only tries to throw a pebble in this way.  

The complex matrix normal distribution is crucial in the works of \cite{dyson1962statistical} due to physical reasoning. For the origin of these four simultaneous diagonalisations $T_1,T_{1\frac{1}{2}},T_2,$ and $T_3$, standard references such as \cite{Pajor2009OnTL} considered the sum of rank one matrices, representative for $T_{2}$, and \cite{Mei2021OnSV} studied the singular value decomposition of matrices with general independent columns, viewed as $T_{1\frac{1}{2}}$. However, as seen in Theorem \ref{lem: dykstral70}, linear independence in random vectors differs from that in constant ones, so for $T_1$, it may be seen as the most general form for the spectral decomposition by similar Dykstral arguments. 



Moving from basic hypergeometric functions to Macdonald and Koornwinder polynomials involves a shift toward $(q,t)$-deformations and more structured symmetries. The key difficulty is a change in the root system as described by \cite{Opdam1989}, which causes the integration measure to deform in a very algebraic way. Geometric explanation for zonal polynomials in the real, complex, or quaternion domains $\mathbb R, \mathbb C, \mathbb H$ tends to break down when the measure gains extra $(q,t)$-parameters. To build a useful configuration space here, one usually has to set aside direct geometry and work inside affine Hecke algebras instead, like \cite{macdonald2003affine}. Thus, there are also interesting links between the statistics and the representation theory.


\bibliographystyle{abbrvnat}
\bibliography{bibtex}

\end{document}